\documentclass[11pt]{article}
\usepackage{amsfonts}
\usepackage{amsmath, amsthm}
\usepackage{amssymb}
\usepackage{eucal}
\newtheorem{theorem}{Theorem}[section]

\newtheorem{conjecture}[theorem]{Conjecture}

\newtheorem{example}[theorem]{Example}

\newtheorem{lemma}[theorem]{Lemma}

\begin{document}

\baselineskip=16pt
\title{On some conjectures related to finite nonabelian simple groups\thanks{The project
was partially supported by NSFC (11671063) and the Natural Science
Foundation of CSTC (cstc2018jcyjAX0060).}}

\author{Jinbao Li$^a$, Wujie Shi$^{a, b,}$\thanks{Corresponding author.}\\
{\small $^a$Department  of Mathematics, Chongqing University of
Arts and Sciences}\\
{\small  Chongqing 402160, P. R. China}\\
\\
{\small $^b$School of Mathematics, Suzhou University}\\
{\small Suzhou, 215006, P. R. China}\\
{\small\  E-mail: leejinbao25@163.com, shiwujie@outlook.com}}
\date{}
\maketitle
\begin{abstract} In this note we provide some counterexamples for the conjecture of Moret\'{o} on
finite simple groups, which says that any finite simple group  $G$
can determined in terms of  its order $|G|$ and the number of
elements of order $p$, where $p$ the largest prime divisor of $|G|$.
Moreover, we show that this conjecture holds
for all sporadic simple groups and alternating groups $A_n$, where $n\neq 8, 10$.  Some related conjectures
are also discussed. \\

{\bf Keywords:} Finite simple groups, quantitative characterization, the largest prime divisor.\\

{\bf AMS Mathematics Subject Classification(2010):} 20D05 20D60.
\end{abstract}

\section{Introduction}

All groups considered in this paper are finite and simple groups
always means finite  nonabelian simple groups.

For a finite group $G$, let $\pi(G)$ be the set of prime divisors
dividing the order of $G$ and let $|G|$ and $\pi_e(G)$ denote the
order of $G$ and the set of element orders of $G$, respectively. For
a positive number $k$, denote by $G(k)$ the set of elements of order
$k$ in $G$. It is well known that $|G|$ and $\pi_e(G)$ are two of
the most important quantitative invariants for $G$.  In 1987, the
second author  posed the following conjecture (see \cite{Shi89}).

\begin{conjecture}
Let $G$ be a group and $S$ be a finite  simple group. Then $G\cong
S$ if and only if $\pi_e(G)=\pi_e(S)$ and $|G| = |S|$.
\end{conjecture}


J.G. Thompson said that \emph{this would certainly be a nice
theorem} (see \cite{Thompson}, personal communication, January 4,
1988) if Shi's conjecture is true. From 1987 to 2003, the authors of
\cite{Cao, Shi89, Shi90, Shi91, Shi92, Shi94, Xu} proved that this
conjecture is correct for all finite simple groups except $B_n$,
$C_n$ and $D_n$ ($n$ even). At the end of 2009, the authors of
\cite{Vasi} proved that this conjecture is correct for $B_n$, $C_n$
and $D_n$ ($n$ even). Thus, this conjecture has been proved and
becomes  a theorem, that is, all finite simple groups can be
determined by their orders and the sets of their element orders
(briefly, `two orders').

Recently, Moret\'{o} in \cite{Mo} took  a new perspective on the
characterization of simple groups and investigated the influence of
the number of elements of order $p$ in  a given group $G$, where
$p\in\pi(G)$. He showed that $A_p$ and $L_2(p)$ are basically
determined just by the number of elements of order $p$, where $p$ is
the largest prime divisor of the order of the group satisfying some
additional conditions. Furthermore, Moret\'{o} posed the following
conjecture.

\begin{conjecture} \label{MC}
Let $S$ be a finite simple group and $p$ the largest prime divisor
of $|S|$. If $G$ is a finite group with the same number of elements
of order $p$ as $S$ and $|G| = |S|$, then $G \cong S$.
\end{conjecture}

In this note we first provide two counterexamples for the above
conjecture.

\begin{example}
Let $S = A_8 \cong L_4(2)$ and $G=L_3(4)$. Then we have $|S|=|G|$
and $7$ is the largest prime in $\pi(S)=\pi(G)$. By {\rm
\cite{CCNPW}}, we have that both $S$ and $G$ contain $2^7\cdot
3^2\cdot 5$ elements of order $7$. But $L_3(4)$ is not isomorphic to
$A_8$.
\end{example}

\begin{example}
Let $S = O_7(3)$ and $G = S_6(3)$. We have $|S| = |G| = 2^9\cdot
3^9\cdot 5\cdot 7\cdot 13$ and $p = 13$ is the largest prime
dividing both the orders of $S$ and $G$. Similarly,  we may conclude
from {\rm \cite{CCNPW}} that $G$ and $S$ have the same number of
elements of order $\mathrm{13}$.
\end{example}


Although  Moret\'{o}'s conjecture is not true in general, we prove
that this conjecture holds for all sporadic simple groups and almost
all alternating groups. Our  main result is  as follows.

\begin{theorem} \label{Main}
Let $G$ be a group, $S$ be any sporadic simple group and $p$ be the
largest prime in $\pi(S)$. If $|G|=|S|$ and $|G(p)|=|S(p)|$, then
$G\cong S$.
\end{theorem}

We also conclude  that this conjecture is valid for alternating
groups $A_n$ except $n=8, 10$ by the following result of Bi in
\cite{Bi}. For a group $G$, Bi proved that if $|G|=|A_n|$ and
$N_G(R)$ and $N_{A_n}(S)$ have the same order, where $R$ and $S$ are
the Sylow $p$-subgroups of $G$ and $A_n$ respectively with $p$ the
largest prime not exceeding $n$ and $n\neq 8, 10$, then $G\cong A_n$
(see \cite[Theorem 1.2]{Bi}).

\begin{theorem} \label{alt}
Suppose that a group $G$ fulfills the conditions of Conjecture
\ref{MC} with $S=A_n$ for $n\geq 5$. Then

{\rm (1)} If $n\neq 8, 10$, then $G\cong S$.

{\rm (2)} If $S=A_8$, then $G\cong A_8$ or $L_3(4)$.

{\rm (3)} If $S=A_{10}$, then $G\cong A_{10}$ or $G\cong J_2\times
Z(G)$, where $Z(G)$ is a cyclic group of order $3$.
\end{theorem}

Note that in Conjecture \ref{MC}, for a given simple $S$, the
largest prime $p$ in $\pi(S)$ and the number of elements of order
$p$ in $S$ play an important role in recognition of $S$. On the
other hand, almost 40 years ago, Herzog in \cite{Herzog}
investigated the influence of the number  of involutions on the
structure of simple groups and proved that many classes of simple
groups are characterized by the numbers of  their involutions.
Furthermore, Herzog conjectured in \cite{Herzog} that if two simple
groups have the same number of involutions, then they have the same
order. However, this conjecture is not true in general and in
\cite{Zarrin}, Zarrin provided a counterexample as follows. Let
$G=L_3(4)$ and $S=S_4(3)$. Then, with notation as above, we have
$|G(2)|=|S(2)|=315$, i.e, both $L_3(4)$ and $S_4(3)$ have 315
elements of order $2$. Zarrin also in \cite{Zarrin} put forward the
following conjecture: If $S$ is a nonabelian simple groups and $G$
is  a groups such that $|G(2)|=|S(2)|$  and $|G(p)|=|S(p)|$ for some
odd prime divisor $p$, then $|G|=|S|$. Later on, Zarrin's conjecture
and its related topics were discussed in Anabanti \cite{Anab} and
Malinowska \cite{Malin} respectively. In \cite{Anab}, Anabanti
disproved Zarrin's conjecture by the following counterexample. Let
$G=L_4(3)$ and $S=L_3(9)$. Then $|G(2)|=|S(2)|$ and
$|G(13)|=|S(13)|$, but it is clear that $|G|\neq |S|$.

Now it is natural to ask whether any  nonabelian simple group can be
determined by the conditions in Moret\'{o}'s conjecture together
with the number of involutions. However, this does not hold in
general. For example, take $G=L_3(4)$ and $S=L_4(2)$, which are not
isomorphic. Then $|G|=|S|$, $|G(2)|=|S(2)|$ and $|G(7)|=|S(7)|$,
where $7$ is the largest prime divisor dividing the orders of $G$
and $S$. We conclude this section with the following conjecture,
which strengthens the conditions of Moret\'{o}'s conjecture.

\begin{conjecture}
Let $G$ be a group and $S$ be a finite simple group. Then $G\cong S$
if and only if $\mathrm{npe}(G)=\mathrm{npe}(S)$ and $|G|=|S|$.
\end{conjecture}

Here we let $\mathrm{npe}(G)=\{|G(p)|\mid p\in \pi(G)\}$ denote the
set of numbers of elements of prime order  in a group $G$.

\section{Preliminary}

In this section, we collect some elementary facts which are useful
in our proof.

For a group $G$, define its prime graph $\Gamma(G)$ as follows: the
vertices are the primes dividing the order of $G$, two vertices $p$
and $q$ are joined by an edge if and only if $G$ contains an element
of order $pq$ (see \cite{Wil}). Denote the connected components of
the prime graph by $T(G)=\{\pi_i(G)|1\leqslant i\leqslant t(G)\}$,
where $t(G)$ is the number of the prime graph components of $G$. If
the order of $G$ is even, assume that the prime 2 is always
contained in $\pi_1(G)$. A simple group whose order has exactly $n$
distinct primes is called a simple $K_n$-group. In addition, for a
group $G$, we call $G$ a $2$-Frobenius group if $G$ has a normal
series $1\subseteq H\subseteq K\subseteq G$ such that $K$ and $G/H$
are Frobenius groups with kernels $H$ and $K/H$ respectively. For
other notation and terminologies mentioned in this paper, the reader
is referred to  \cite{CCNPW} if necessary.

\begin{lemma} \label{non-con}
Let $G$ be a group with more than one prime graph component. Then
$G$ is one of the following:

{\rm (i)} a Frobenius or $2$-Frobenius group;

{\rm (ii)} $G$ has a normal series $1\subseteq H\subseteq K\subseteq
G$, where $H$ is a nilpotent $\pi_1$-group, $K/H$ is a nonabelian
simple group and $G/K$ is a $\pi_1$-group such that $|G/K|$ divides
the order of the outer automorphism group of $K/H$. Besides, for
$i\geq 2$, $\pi_i(G)$ is also a component of $\Gamma(K/H)$.
\end{lemma}

\begin{proof}
It follows straight forward from  Lemmas 1-3 in \cite{Wil}, Lemma
1.5 in \cite{Ch2} and Lemma 7 in \cite{Ch4}.
\end{proof}

\begin{lemma} \label{Fro}
Suppose that $G$ is a Frobenius group of even order and $H$, $K$ are
the Frobenius kernel and the Frobenius complement of $G$,
respectively. Then $t(G)=2$, $T(G)=\{\pi(H), \pi(K)\}$ and $G$ has
one of the following structures:

{\rm (i)} $2\in \pi(H)$ and all Sylow subgroups of $K$ are cyclic;

{\rm (ii)} $2\in \pi(K)$, $H$ is an abelian group, $K$ is a solvable
group, the Sylow subgroups of $K$ of odd order are cyclic groups and
the Sylow $2$-subgroups of $K$ are cyclic or generalized quaternion
groups.
\end{lemma}

\begin{proof}
This is Lemma 1.6 in \cite{Ch3}.
\end{proof}

\begin{lemma} \label{2Fro}
 Let $G$ be a $2$-Frobenius group
of even order. Then $t(G)=2$ and $G$ has a normal series $1\subseteq
H\subseteq K\subseteq G$ such that $\pi(K/H)=\pi_2$, $\pi(H)\cup
\pi(G/K)=\pi_1$, the order of $G/K$ divides the order of the
automorphism group of $K/H$, and both $G/K$ and $K/H$ are cyclic.
Especially, $|G/K|<|K/H|$ and $G$ is soluble.
\end{lemma}

\begin{proof}
This is Lemma 1.7 in \cite{Ch3}.
\end{proof}

The following lemma is well-known (see \cite[Theorem 3.3.20]{DJS}).
\begin{lemma} \label{Aut}
Let $R=R_1\times \cdots\times R_k$, where $R_i$ is a direct product
of $n_i$ isomorphic copies of a nonabelian simple group $H_i$ and
$H_i$ and $H_j$ are not isomorphic if $i\neq j$. Then
\begin{center}
{\rm Aut($R$)$\cong$ Aut($R_1$)$\times \cdots\times$ Aut($R_k$) and
Aut($R_i$)$\cong$ (Aut($H_i$))$\wr S_{n_i}$.}
\end{center}
Moreover,
\begin{center}
{\rm Out($R$)$\cong$ Out($R_1$)$\times \cdots\times$ Out($R_k$) and
Out($R_i$)$\cong$ (Out($R_i$))$\wr S_{n_i}$.}
\end{center}
\end{lemma}

\section{Proof of main results}

In this section, we first present the proof of Theorem \ref{Main}.

\noindent {\emph{Proof of Theorem \ref{Main}.}}

 We will proceed our
proof case by case. Suppose that $G$ and $S$ satisfy the hypothesis
of Theorem \ref{Main}. Then by \cite{CCNPW}, one can compute the
number of elements of order $p$, denoted by $|G(p)|$, where $p$ is
the largest prime in $\pi(G)=\pi(S)$. Let $P$ be a Sylow
$p$-subgroup of $G$. Then $P$ is of order $p$ and
$$|G:N_G(P)|=\frac{|G(p)|}{p-1}.$$Therefore we know the order of
$N_G(P)$. From now on, we distinguish the following several cases.

(1) If $S=M_{11}$, then $G\cong S$.

By the hypothesis, $|G|=2^4\cdot 3^2\cdot 5\cdot 11$ and
$|G(11)|=2^5\cdot 3^2\cdot 5$. Let $P$ be a Sylow $11$-subgroup of
$G$. Then $|N_G(P)|=5\cdot 11$. Since $C_G(P)\leq N_G(P)$, we have
that either $N_G(P)=C_G(P)$ or $C_G(P)$ is a cyclic group of order
11. Suppose that $N_G(P)=C_G(P)$. Then $G$ has a normal
$11$-complement $K$ in $G$. Let $Q$ be a Sylow 3-subgroup of $K$
such that $PQ$ is a Hall subgroup of $G$. Let $N$ be a minimal
normal subgroup in $PQ$. If $|N|=11$, then $P$ is normal in $PQ$ and
so $Q\leq N_G(P)$, a contradiction. If $|N|=3$ or $3^2$, then $N\leq
C_G(P)$, which is impossible. Hence,  we must have $C_G(P)=P$.

Therefore, $\Gamma(G)$ is not connected and $\{11\}$ is a component
of $\Gamma(G)$. Thus, by Lemma \ref{non-con}, $G$ is  a Frobenius
group or a 2-Frobenius group or a group satisfying the conditions in
(2) of Lemma \ref{non-con}. Suppose that $G$ is a Frobenius group
with kernel $H$. Then by Lemma \ref{Fro}, we have $|H|=2^4\cdot
3^2\cdot 5$ and so $G$ has a Hall subgroup of order $3^2\cdot 11$, a
contradiction as above. If $G$ is a 2-Frobenius group, then $G$ is
solvable by Lemma \ref{2Fro} and so $G$ also has a subgroup with
order $3^2\cdot 11$, a contradiction.

Therefore, by Lemma \ref{non-con}, we have that $G$ has a normal
series $1\subseteq H\subseteq K\subseteq G$, where $H$ is a
nilpotent $\pi_1$-group, $K/H$ is a nonabelian simple group and
$G/K$ is a $\pi_1$-group such that $|G/K|$ divides the order of the
outer automorphism group of $K/H$. Besides, for $i\geq 2$,
$\pi_i(G)$ is also a component of $\Gamma(K/H)$. Hence, $11$ divides
the order of $K/H$. By \cite[Table 1]{Zava}, we have that $K/H$ is
isomorphic to $L_2(11)$ or $M_{11}$.

Suppose that $K/H\cong L_2(11)$. Since $G/K$ divides the order of
the outer automorphism group of $K/H$ and $|K/H|=2^2\cdot 3\cdot
5\cdot 11$, it follows that $2\cdot 3$ divides the order of $H$ and
therefore $3$ divides the order of $N_G(P)$, a contradiction. Hence,
$K/H\cong M_{11}$ and so $G\cong M_{11}$.

(2) If $S=M_{12}$, then $G\cong S$.

By the hypothesis, $|G|=2^6\cdot 3^3\cdot 5\cdot 11$ and
$|G(11)|=2^7\cdot 3^3\cdot 5$. Let $P$ be a Sylow $11$-subgroup of
$G$. Then $N_G(P)$ is of order $5\cdot 11$. Using a similar argument
as in case (1), we have that $G$ has a normal series $1\subseteq
H\subseteq K\subseteq G$, where $H$ is a nilpotent $\pi_1$-group,
$K/H$ is a nonabelian simple group and $G/K$ is a $\pi_1$-group such
that $|G/K|$ divides the order of the outer automorphism group of
$K/H$. Besides, for $i\geq 2$, $\pi_i(G)$ is also a component of
$\Gamma(K/H)$. Hence, $11$ divides the order of $K/H$. By
\cite[Table 1]{Zava}, we have that $K/H$ is isomorphic to $L_2(11)$,
$M_{11}$ or $M_{12}$. It is easy to check that $K/H$ must isomorphic
to $M_{12}$ and so is $G$, as desired.

(3) If $S=M_{22}, M_{23}, M_{24}, HS, McL, Co_3, Co_2, Fi_{23},
Co_1$, then it follows form \cite{CCNPW} that the order of $N_G(P)$
is either $5\cdot 11$ or $11\cdot 23$. Hence we obtain similarly
that $G\cong S$ in these cases.

(4) If $S=J_2$, then $G\cong S$.

By \cite{CCNPW}, $|G|=2^7\cdot 3^3\cdot 5^2\cdot 7$ and
$|G(7)|=2^7\cdot 3^3\cdot 5^2$. It follows that if $P$ is a  Sylow
$7$-subgroup of $G$, then $|N_G(P)|=2\cdot 3\cdot 7$. Let $K$ be the
largest normal solvable subgroup of $G$. Then $G\neq K$. Otherwise,
if $G$ is a solvable group, then $G$ has  a Hall subgroup $QP$ of
order $5^2\cdot 7$, where $Q$ is a Sylow $5$-subgroup of $G$. Pick a
minimal normal subgroup $N$ of $QP$. If $N=P$, then $Q\leq N_G(P)$,
which is impossible. If $N$ is a cyclic group of order 5 or an
elementary abelian group of order 25, then $P$ acts trivially on $N$
and therefore $N\leq N_G(P)$, another contradiction.  Hence $K$ is a
proper subgroup of $G$.

Furthermore, we assert that $7\notin \pi(K)$. If $7\in \pi(K)$, then
$P$ is a Sylow $7$-subgroup of $K$ and so $G=N_G(P)K$. By the order
of $N_G(P)$, $N_G(P)$ is a solvable group, which implies that $G$ is
also solvable, a contradiction. Hence $7\notin \pi(K)$. Let $N$ be a
minimal normal subgroup of $G/K$. Then $N$ is a direct product of
nonabelian simple groups and without loss of generality, we can
assume that $7\in \pi(N)$. Otherwise, if $7$ does not divide the
order of any minimal normal subgroup of $G/K$, then $2^{14}$ will
divide $|G|$ by Lemma \ref{Aut}, a contradiction by our hypothesis.
Moreover we have that $N$ is a simple group. By \cite[Table
1]{Zava}, we obtain that $N$ is isomorphic to $$J_2, A_8, L_3(4),
U_3(3), L_2(8), L_2(7).$$

If $N\cong A_8$, then $N$ is the unique minimal normal subgroup of
$G/K$ by the orders of $G$ and $N$ and the above arguments. It
follows that $5\in \pi(K)$. Take a Hall $\{5, 7\}$-subgroup $PK_5$
of $PK$, where $K_5$ is a Sylow $5$-subgroup of $K$. Then $K_5\leq
C_G(P)$, a contradiction.

Similarly, $N$ is not isomorphic to $L_3(4)$.

Suppose that $N\cong A_7$. If $N$ is the unique normal subgroup of
$G/K$, then $5\in \pi(K)$, a contradiction. If there exists another
minimal normal subgroup in $G/K$, say $L$, then by the order of $G$,
we know $5$ divides the order of $N_{G/K}(PK/K)$. But
$N_{G/K}(PK/K)=N_G(P)K/K$ since $7\notin \pi(K)$, a contradiction.

Similarly, we have that $N$ is not isomorphic to $U_3(3), L_2(8),
L_2(7)$. Therefore, $N$ must be isomorphic to $J_2$ and so $G\cong
J_2$, as desired.

(5) If $S=He$, then $G\cong S$.

By the hypothesis, we have that $|G|=2^{10}\cdot 3^3\cdot 5^2\cdot
7^3\cdot 17$ and $|G(17)|=2^{11}\cdot 3^3\cdot 5^2\cdot 7^3$ for
$p=17$. Thus $|N_G(P)|=2^3\cdot 17$, where $P$ is a Sylow
$17$-subgroup of $G$. Let $K$ be the largest solvable normal
subgroup in $G$. Similar to case (12), we can deduce that $G\neq K$
and $G/K$ has a minimal normal subgroup $S$ such that $17\in
\pi(S)$. Then, clearly,  $S$ is a nonabelian simple group and by
\cite[Table 1]{Zava}, we see that $S$ is isomorphic to one of the
following groups:$$He, S_4(4), L_2(16), L_2(17).$$ If $S$ is not
isomorphic to $He$, then $3$ or $3^2$ divides the order of $N_G(P)$,
which is impossible. Hence $S\cong He$ and so the result follows.

(6) $S=Suz$.

By \cite{CCNPW}, $|G|=2^{13}\cdot 3^7\cdot 5^2\cdot 7\cdot 11\cdot
13$ and $|G(13)|=2^{14}\cdot 3^7\cdot 5^2\cdot 7\cdot 11$. It
follows that $|N_G(P)|=2\cdot 3\cdot 13$, where $P$ is a Sylow
$13$-subgroup of $G$. Let $\pi^1=\{7, 11, 13\}$. We claim that $G$
has a chief factor $H/K$ such that $\pi^1\subseteq \pi(H/K)$. Let
$H/K$ be a chief factor of $G$ such that $\pi(H/K)\cap \pi^1\neq
\emptyset$ and $\pi(K)\cap \pi^1=\emptyset$. We first show that $13$
divides the order of $H/K$. If not, then $13$ divides the order of
$G/H$. Since $\pi(H/K)\cap \pi^1\neq \emptyset$, $7\in \pi(H/K)$ or
$11\in \pi(H/K)$. If $7\in \pi(H/K)$, then $G=N_G(G_7)H$, where
$G_7$ is a Sylow 7-subgroup of $G$. It follows that $13$ divides the
order of $N_G(G_7)$ and so $7$ divides the order of $N_G(P)$, a
contradiction. Similarly we get that $11$ divides the order of
$N_G(P)$ if $11\in \pi(H/K)$. Hence $13\in \pi(H/K)$. Discussing
repeatedly  as above, we obtain that both 7 and 11 are contained in
$\pi(H/K)$. It is easy to see that $H/K$ is a nonabelian simple
group. By \cite[Table 1]{Zava}, $H/K$ is isomorphic to $Suz$ or
$A_{13}$.

Suppose that $H/K\cong A_{13}$. Then $|H/K|=|A_{13}|=2^9\cdot
3^5\cdot 5^2\cdot 7\cdot 11\cdot 13$ and thus $|K|$ divides
$2^4\cdot 3^2$. Let $N$ be a minimal normal subgroup of $G$
contained in $K$. Then $N$ is an elementary group of order $2^i$ or
$3^j$ with $i=1, 2, 3, 4$ and $j=1, 2$. If $|N|=3^2$ or $2^i$ with
$2\leq i\leq 4$, then $P$ acts trivially on $N$, and therefore
$N\leq N_G(P)$, a contradiction. Thus, $|G/H|\geq 2^3\cdot 3$. Note
that $N_{G/K}(PK/K)=N_G(P)K/K$ since $|N|$ is coprime to 13 and
$|N_{H/K}(PK/K)|=13(13-1)/2$. It follows that $H/K$ is the unique
minimal normal subgroup of $G/K$ and consequently $|G/H|$ divides
the order of outer automorphism group of $H/K$, a contradiction.
Hence $H/K\cong Suz$ and thus $G\cong Suz$.

(7) $S=J_1$, then $G\cong S$.

In this case, $|G|=2^3\cdot 3\cdot 5\cdot 7\cdot 11\cdot 19$ and
$|G(19)|=2^3\cdot 3^2\cdot 5\cdot 7\cdot 11$ for $p=19$. It follows
that $N_G(P)$ is of order $2\cdot 3\cdot 19$ with $P$ a Sylow
$19$-subgroup of $G$. Set $\pi^1=\{5, 7, 11, 19\}$. Then, arguing as
in the case (14), one have that $G$ has a chief factor $H/K$ such
that $\pi^1\subseteq \pi(H/K)$. It follows from \cite[Table 1]{Zava}
that $H/K\cong J_1$ and so $G$ must be isomorphic to $J_1$.

(8) If $S=$$J_3$, then $G\cong S$.

By the hypothesis, we have that $|G|=2^7\cdot 3^5\cdot 5\cdot
17\cdot 19$ and $|G(19)|=2^8\cdot 3^5\cdot 5\cdot 17$ for $p=19$.
Then $N_G(P)$ is of order $3^2\cdot 19$. Let $\pi^1=\{5, 17, 19\}$.
Then, similarly as above,  $G$ has a chief factor $H/K$ such that
$\pi^1\subseteq \pi(H/K)$. By \cite[Table 1]{Zava}, $H/K$ must be
isomorphic to $J_3$ and so the result follows.

(9) If $S=Ru$, then $G\cong S$.

In this case, we have that $|G|=2^{14}\cdot 3^3\cdot 5^3\cdot 7\cdot
13\cdot 29$ and $|G(29)|=2^{15}\cdot 3^3\cdot 5^3\cdot 7\cdot 13$
for $p=29$. Let $P$ be a Sylow $29$-subgroup of $G$. Then
$|N_G(P)|=2\cdot 7\cdot 29$. Take $\pi^1=\{13, 29\}$. Then the
result follows similarly.

(10) If $S=O'N$, then $G\cong S$.

By the hypothesis, we have that $|G|=2^9\cdot 3^4\cdot 5\cdot
7^3\cdot 11\cdot 19\cdot 31$ and $|G(31)|=2^{10}\cdot 3^4\cdot
5\cdot 7^3\cdot 11\cdot 19$. Then $|N_G(P)|=3\cdot 5\cdot 31$ for
some Sylow $31$-subgroup $P$ of $G$.  In this case, if we let
$\pi^1=\{11, 19, 31\}$, then the result follows similarly.

(11) If $S=Fi_{22}, HN, Ly, Th, J_4, Fi_{24}', B, M$, then we can
pick appropriate $\pi^1$ for these simple groups  and check as
above.















By the foregoing arguments, we obtain that Conjecture \ref{MC} is
true for all sporadic groups and Theorem \ref{Main} is thus proved.
\qed

Now we proceed  to prove Theorem \ref{alt}.

\noindent {\emph{Proof of Theorem \ref{alt}.}}

(1) Let $G$ be a group satisfying the hypothesis and $p$ be the
largest prime dividing the order of $A_n$ with $n\neq 8, 10$. Then
$|G|_p=|A_n|_p=p$ and $|G(p)|=|A_n(p)|$. Therefore $G$ and $A_n$
have the same number of Sylow $p$-subgroups and so $G\cong A_n$ by
Bi's result in \cite{Bi}.

(2) Suppose that $|G|=|A_8|$ and $|G(7)|=|A_8(7)|=2^7\cdot 3^2\cdot
5$. Let $P$ be a Sylow $7$-subgroup of $G$. Then $N_G(P)$ is of
order $3\cdot 7$. If $N_G(P)=C_G(P)$, then $G$ has a normal
$7$-complement $K$. Then it is easy to see for some Sylow
$5$-subgroup $Q$ of $K$, $Q\leq N_G(P)$, a contradiction. Hence,
$C_G(P)=P$, which implies that $\Gamma(G)$ is not connected and
$\{7\}$ is a component in $\Gamma(G)$. By Lemmas \ref{non-con},
\ref{Fro} and \ref{2Fro}, we conclude that $G$ has a normal series
$1\subseteq H\subseteq K\subseteq G$, where $H$ is a nilpotent
$\pi_1$-group, $K/H$ is a nonabelian simple group and $G/K$ is a
$\pi_1$-group such that $|G/K|$ divides the order of the outer
automorphism group of $K/H$. Besides, for $i\geq 2$, $\pi_i(G)$ is
also a component of $\Gamma(K/H)$. In particular, $7$ is a component
of $\Gamma(K/H)$. By \cite[Table 1]{Zava}, we see that $K/H$ is
isomorphic to one of
$$A_8, L_3(4), A_7, L_2(8), L_2(7).$$

Suppose that $K/H\cong A_7$. Then $|K/H|=2^3\cdot 3^2\cdot 5\cdot 7$
and so $2^2$ divides the order of $H$. Clearly, $P$ acts trivially
on $H$ and it follows that $H\leq C_G(P)\leq N_G(P)$, a
contradiction.

If $K/H\cong L_2(8)$ or $L_2(7)$, then $5$ does not divides the
order of $\mathrm{Aut}(K/H)$ by \cite{CCNPW} and so $5$ divides the
order of $H$. Let $H_5$ be the Sylow $5$-subgroup of $H$. Then
$H_5\leq C_G(P)$ as above.

Therefore $K/H\cong A_8$ or $L_3(4)$ and consequently $G\cong A_8$
or $L_3(4)$.

(3) Suppose that $|G|=|A_{10}|$ and $|G(7)|=|A_{10}(7)|=2^7\cdot
3^3\cdot 5^2$. Then $|N_G(P)|=2\cdot 3^2\cdot 7$ with $P$ a Sylow
$7$-subgroup of $G$. Arguing as in the proof of case (3) of Theorem
\ref{Main}, we get that if $K$ is the largest solvable normal
subgroup of $G$, then $G/K$ has a minimal normal subgroup $N$ such
that $7\in \pi(N)$. It is clear that $N$ is a simple group. By
\cite[Table 1]{Zava}, we see that $N$ is isomorphic to one of
$$A_{10}, J_2, A_8, L_3(4), A_7, U_3(3), L_2(8), L_2(7).$$

Since $N_{G/K}(PK/K)=N_G(P)K/K$, we see that $N$ is the unique
minimal normal subgroup of $G/K$. Otherwise, $|N_{G/K}(PK/K)|\geq
2^2\cdot 3\cdot 5\cdot 7$, a contradiction. Let $N=H/K$. Then we
have that $|G/H|$ divides the order of outer automorphism group of
$N$.

If $N\cong A_8$, then $5$ divides the order of $K$ and so one can
deduce that $5$ divides $|N_G(P)|$, a contradiction.

If $N\cong L_3(4)$, then $5$ divides the order of $K$ since
$|\mathrm{Out}(L_3(4))|=12$ and so we have a contradiction.

If $N\cong A_7, U_3(3), L_2(8), L_2(7)$, we also have that $5$
divides $|K|$. Hence we obtain that $N\cong J_2$ or $A_{10}$.
Suppose that $N\cong J_2$. Then $|N|=2^7\cdot 3^3\cdot 5^2\cdot 7$
and $G/K=N$ since $|\mathrm{Out}(J_2)|=2$. It follows that $|K|=3$.
Clearly, $C_G(K)>K$, which forces $C_G(K)=G$ and so $K=Z(G)$. Since
$G=G'Z(G)$, we have $G'\cap Z(G)=Z(G)$ or 1. But the Schur
multiplier of $J_2$ is a cyclic group of order 2. Hence $G'\cap
Z(G)=1$ and so $G\cong J_2\times Z_3$, where $Z_3$ is a cyclic group
of order 3. By \cite{CCNPW}, we get that $|J_2(7)|=2^7\cdot 3^3\cdot
5^2$ and so $|G(7)|=2^7\cdot 3^3\cdot 5^2$, as wanted. At last, if
$N\cong A_{10}$, then $G\cong A_{10}$ as well.

Thus, the proof is complete.\qed


\begin{thebibliography}{10}


\bibitem{Anab} C.S. Anabanti, A counterexample to Zarrin's conjecture on sizes of
finite nonabelian simple groups in relation to involution sizes,
Arch. Math., 112, 2019, 225-226.


\bibitem{Bi} J.X. Bi, Characterization of alternating groups by
orders of normalizers of Sylow subgroups, Algebra Colloquium, 8(3),
2001, 249-256.



\bibitem{Cao}
H.P. Cao, W.J. Shi, Pure quantitative characterization of finite
projective special unitary groups, Sci. China, Ser. A, 45,  2002,
761-772.

\bibitem{Ch2} G.Y. Chen, On Thompson's cconjecture for sporadic
simple groups, Proc. China Assoc. Sci. and Tech. First Academic
Annual Meeting of Youths, pp.1-6, Chinese Sci. and Tech. Press,
Beijing, 1992. (in Chinese)

\bibitem{Ch3} G.Y. Chen, On Thompson's conjecture, J. Algebra, 185,
1996, 184-193.

\bibitem{Ch4} G.Y. Chen, Further reflections on Thompson's
conjecture, J. Algebra, 218, 1999, 276-285.

\bibitem{CCNPW} J.H. Conway, R.T. Curtis, S.P. Norton, R.A.
Parker, R.A. Wilson, Atlas of Finite Groups, Clarendon Press,
Oxford, 1985.

\bibitem{Herzog} M. Herzog, On the classification of finite simple
groups by the number of involutions, Proc. American Math. Soc.,
77(3), 1979, 313-314.

\bibitem{Malin} I.A. Malinowska, Finite groups with few normalizers or
involutions, Arch. Math., 112, 2019, 459-465.

\bibitem{Mo} A. Moret\'{o}, The number of elements of prime order, Monatsh. Math.,  186, 2018, 189-195.



\bibitem{DJS} Derek J.S. Robinson, A Course in the Theory of Groups,
Springer-Verlag, New York-Heidelberg-Berlin, 1982.




\bibitem{Shi89}
W.J. Shi, A new characterization of the sporadic simple groups,
Group Theory - Proc. Singapore Group Theory Conf. 1987, Walter de
Gruyter,  Berlin-New York, 1989£¬ 531-540.

\bibitem{Shi90} W.J. Shi, J.X.  Bi, A characteristic property for each finite projective special
linear group (with J.X. Bi), Lecture Notes in Math.,
Springer-Verlag, 1456, 1990, 171-180.

\bibitem{Shi91} W.J. Shi, J.X. Bi, A characterization of Suzuki-Ree groups, Sci. in China, Ser. A,
34,  1991, 14-19.

\bibitem{Shi92} W.J. Shi, J.X. Bi, A new characterization
of the alternating groups, Southeast Asian Bull. Math., 16,  1992,
81-90.

\bibitem{Shi94} W.J. Shi, The pure quantitative characterization of
finite simple groups (I), Prog. Nat. Sci., 4,  1994, 316-326.


\bibitem{Thompson} J.G. Thompson, personal communication, January 4,
1988.

\bibitem{Vasi} A.V. Vasilev, M.A.
Grechkoseeva, and V.D. Mazurov, Characterization of the finite
simple groups by spectrum and order, Algebra and Logic, 48, 2009,
385-409.



\bibitem{Wil} J. S. Williams, Prime graph components of finite groups,
J. Algebra, 69, 1981, 487-513.

\bibitem{Xu} M.C. Xu and W.J. Shi, Pure quantitative
characterization of finite simple groups $^2D_n(q)$ and $D_l(q)$
($l$ odd), Alg. Coll., 10,  2003, 427-443.

\bibitem{Zarrin} M. Zarrin, A counterexample to Herzog's Conjecture on the number of
involutions, Arch. Math., 111, 2018, 349-351.


\bibitem{Zava} A.V. Zavarnitsine, Finite simple groups with narrow
prime spectrum, Siberian Electronic Math. Reports, 6, 2009, 1-12.


\end{thebibliography}
\end{document}